\documentclass[a4paper,12pt,reqno]{amsart}
\usepackage{amsfonts}
\usepackage{amsmath}
\usepackage{amssymb}
\usepackage[a4paper]{geometry}
\usepackage{mathrsfs}
\usepackage{csquotes}
\usepackage{enumitem}
\usepackage{dirtytalk}
\usepackage{comment}
\usepackage{nccmath}

\usepackage[colorlinks]{hyperref}
\renewcommand\eqref[1]{(\ref{#1})} 
\numberwithin{equation}{section}
\theoremstyle{plain}
\newtheorem{thm}{Theorem}[section]
\newtheorem{prop}[thm]{Proposition}
\newtheorem{cor}[thm]{Corollary}
\newtheorem{lem}[thm]{Lemma}

\theoremstyle{definition}
\newtheorem{defn}[thm]{Definition}
\newtheorem{rem}[thm]{Remark}
\newtheorem{ex}[thm]{Example}

\def\e[#1]{{\textrm{e}}^{#1}}

\begin{document}

   \title[General weighted Hardy and Rellich type inequalities]
 {Sharp remainder formulae for general weighted Hardy and Rellich type inequalities for $1<p<\infty$} 

   \author[Y. Shaimerdenov]{Yerkin Shaimerdenov}
\address{
  Yerkin Shaimerdenov:
 \endgraf
   SDU University, Kaskelen, Kazakhstan
  \endgraf
  and 
  \endgraf
  Department of Mathematics: Analysis, Logic and Discrete Mathematics 
  \endgraf
  Ghent University, Belgium
  \endgraf
  and
  \endgraf
  Institute of Mathematics and Mathematical Modeling, Kazakhstan
   \endgraf
  {\it E-mail address} {\rm
yerkin.shaimerdenov@sdu.edu.kz}
  }

\author[N. Yessirkegenov]{Nurgissa Yessirkegenov}
\address{
  Nurgissa Yessirkegenov:
  \endgraf
  KIMEP University, Almaty, Kazakhstan
     \endgraf
  {\it E-mail address} {\rm nurgissa.yessirkegenov@gmail.com}
  }

  \author[A. Zhangirbayev]{Amir Zhangirbayev}
\address{
  Amir Zhangirbayev:
 \endgraf
   SDU University, Kaskelen, Kazakhstan
  \endgraf
  and 
  \endgraf
  Institute of Mathematics and Mathematical Modeling, Kazakhstan
   \endgraf
  {\it E-mail address} {\rm
amir.zhangirbayev@gmail.com}
  }

\thanks{This research is funded by the Committee of Science of the Ministry of Science and Higher Education of the Republic of Kazakhstan (Grant No. AP23490970).}


     \begin{abstract} Inspired by the work of Cossetti and D’Arca \cite{d2025unified}, we show that the general weighted $L^{p}$-Hardy type inequalities \cite[Theorems 1.1 and 1.2]{d2025unified} and the corresponding identities hold for all $1<p<\infty$, thus extending their results beyond the case $p\geq 2$. In addition, we present a general weighted $L^{p}$-Rellich type inequality with a sharp remainder term for quasilinear second order degenerate elliptic differential operators. In particular, even for the classical Laplacian, these identities appear to be new.

     \end{abstract}
     \maketitle   

\section{Introduction}

The study of general weighted Hardy type inequalities dates back to the work of Ghoussoub and Moradifam \cite{ghoussoub2011bessel}, where they obtained a necessary and sufficient condition on a pair of positive radial weights $V(|x|)$ and $W(|x|)$ on a ball $B_{R}$ such that the following inequalities hold for all $u\in C_{0}^{\infty}(B_{R})$:
\begin{align*}
\int_{B_{R}}V(|x|)|\nabla u|^{2}dx\geq\int_{B_{R}}W(|x|)|u|^{2}dx.
\end{align*}
Their results improved, extended and unified many Hardy inequalities previously known at that time (e.g., \cite{caffarelli1984first, brezis1997blow, wang2003caffarelli}). Over the next few years, a huge influx of contributions has been made, see, for instance, \cite{ghoussoub2013functional, lam2019factorizations, duy2020improved, lam2020geometric, duy2022p, do2024new, flynn2025p}. 

Another important work that we would like to mention is by Adimurthi and Sekar \cite{sekar2006role}, which might have been one of the first papers where the role of the fundamental solution was studied in the context of Hardy inequalities for general elliptic operators. In particular, they established the following inequality:
\begin{align}\label{AS06 ineq}
\int_{\Omega} |\nabla u|_{A}^{p} dx-\left( \frac{p-1}{p} \right)^{p}
\int_{\Omega} \frac{|\nabla E|_{A}^{p}}{E^{p}} |u|^{p} dx\geq 0,
\end{align}
where $|\xi|_{A}^{2}=\langle A(x)\xi,\xi \rangle$ for $\xi \in \mathbb{R}^{N}$, $A(x)$ is a $N\times N$ symmetric, uniformly positive definite matrix defined on a bounded domain $\Omega \subset \mathbb{R}^N$ and $E$ is a solution to
\begin{align*}
\begin{cases}
-\operatorname{div}\!\left( |\nabla E|_{A}^{\,p-2}\, A \nabla E \right) = \delta_{0} & \text{in } \Omega, \\[0.2cm]
E = 0 & \text{on } \partial\Omega.
\end{cases}
\end{align*}
Here, $0\in \Omega$ and $\delta_{0}$ is a Dirac mass at $0$. Shortly after, the approach showed to be not only applicable to both compact and non-compact manifolds \cite[Section 4.3, 4.4]{sekar2006role}, but also very useful in other settings \cite{frank2008non, goldstein2018general, KY18, yener2018general, yener2018several, cazacu2021method, shaimerdenov2025sharp, d2025unified}. For instance, in \cite{frank2008non}, Frank and Seiringer found out that ground state representations can also be used to obtain a general $L^{p}$-Hardy type inequality and its improvement with a sharp constant. To be precise, suppose $\phi$ is a positive weak solution of the weighted $p$-Laplace equation
\begin{align}\label{fs cond}
-\operatorname{div}\left(V|\nabla \phi|^{p-2} \nabla \phi\right)=W \phi^{p-1}, \quad V>0 \quad \text{on} \quad \mathbb{R}^N. 
\end{align}
Then, for all complex-valued $u$ with compact support, one has
\begin{align}\label{fs ineq}
\int_{\mathbb{R}^{N}}V|\nabla u|^{p}dx\geq \int_{\mathbb{R}^N}W|u|^{p}dx,
\end{align}
where $\int_{\mathbb{R}^{N}}V|\nabla u|^{p}dx$ and $\int_{\mathbb{R}^N}W_{+}|u|^{p}dx$ are assumed to be finite. Interestingly, compared to Bessel pairs in \cite{ghoussoub2011bessel}, the condition (\ref{fs cond}) allows also for non-radial weights $V$ and $W$. Under the restriction to $p\geq2$, due to the convexity inequality, they also obtained an improvement of (\ref{fs ineq}) with a sharp constant \cite[Formula (2.12)]{frank2008non}. 

Continuing from the results of Adimurthi and Sekar, in \cite[Open problem 4.1]{sekar2006role}, Adimurthi and Sekar asked whether the constant $\left(\frac{p-1}{p}\right)^{p}$ in (\ref{AS06 ineq}) is optimal. This problem was later resolved by Cowan \cite{cowan2010optimal}, whose work initiated a broader and deeper analysis of general Hardy inequalities, including the determination of best constants and possible improvements.

Building on this, Ghoussoub and Moradifam, in their book \cite[Chapter 10]{ghoussoub2013functional}, addressed the possibility of obtaining a similar characterization of Hardy inequalities in terms of Bessel pairs (as in \cite{ghoussoub2011bessel}), but for more general elliptic operators \cite[page 156, Open Problem 10]{ghoussoub2013functional}. This question was partially answered  in the same book \cite{ghoussoub2013functional} and was completely resolved in \cite{ruzhansky2024hardy}. For the applications of Bessel pairs in non-Euclidean settings, we refer to \cite{lam2018hardy, roychowdhury2022hardy, ruzhansky2024hardy, d2024unified, GJR24, nguyen2025hardy, shaimerdenov2025sharp, d2025unified}.

In the very recent work \cite{d2025unified}, Cossetti and D'Arca provided similar results in the form of a distributional solution to the degenerate elliptic $p$-Laplacian equation. Their approach brings together the supersolution technique employed in \cite{sekar2006role, frank2008non, ruzhansky2024hardy, cazacu2021method, shaimerdenov2025sharp} and the method introduced by D’Ambrosio \cite{D'A05}, thereby unifying these tools within a common framework. However, their proofs relied on the fundamental algebraic identities \cite[Proposition 2.1 and 2.3]{d2025unified}, which hold for $p\geq2$. 

In this paper, we demonstrate that the equivalent algebraic identity (\ref{cp-eq}), valid for all $1<p<\infty$, allows us to extend the results of \cite{d2025unified} beyond the case $p\geq 2$ to the entire range $1<p<\infty$. We remark that identity (\ref{cp-eq}) has already proved effective in establishing numerous functional inequalities across a wide range of contexts (see, e.g., \cite{duy2022p, do2023p, do2024scale, cazacu2024hardy, kalaman2024cylindrical, CT24, yessirkegenov2025refined, apseit2025sharp, shaimerdenov2025sharp, ruzhansky2025cylindrical}).

Apart from $L^{p}$-Hardy type inequalities, we also establish $L^{p}$-Rellich type inequality with a sharp remainder term. First, let us provide some historical background. 

In the aforementioned paper \cite{ghoussoub2011bessel}, it was shown that the Bessel pairs (along with some additional conditions) not only give general weighted Hardy inequalities, but also general Hardy-Rellich inequalities of the following form: let $n\geq 4$ and $W(|x|)$ be a radial Bessel potential such that $r\frac{W_{r}(r)}{W(r)}$ decreases to $-\lambda$ and if $\lambda \leq n-2$, then for all $u \in C_{0}^{\infty}(B_{R})$, we have
\begin{multline}\label{rellich gm}
\int_{B}|\Delta u|^{2} d x \geq  \frac{n^{2}(n-4)^{2}}{16} \int_{B} \frac{u^{2}}{|x|^{4}} d x \\
 +\left(\frac{n^{2}}{4}+\frac{(n-\lambda-2)^{2}}{4}\right) \beta(W ; R) \int_{B} \frac{W(|x|)}{|x|^{2}} u^{2} d x.
\end{multline}
By choosing appropriate Bessel pairs on $(0,R)$, the inequality (\ref{rellich gm}) improves in various ways the results of Adimurthi, Grossi and Santra \cite{adimurthi2006optimal} as well as Tertikas and Zographopoulos \cite{tertikas2007best}. In addition, in the same paper \cite{ghoussoub2011bessel}, more general and higher-order versions of (\ref{rellich gm}) were obtained that further improve other results from \cite{tertikas2007best}.

Although the result (\ref{rellich gm}) does give the classical Rellich inequality and its various improvements, here, we are more interested in obtaining a result in the form of Frank and Seiringer (\ref{fs ineq}). In \cite{kombe2017weighted}, Kombe and Yener did exactly that: let $V \in C^2\left(\mathbb{R}^N\right)$ and $W \in L_{loc}^1\left(\mathbb{R}^N\right)$ be nonnegative functions. Suppose that $\phi \in C^\infty\left(\mathbb{R}^N\right)$ is a positive function satisfying the differential inequalities
\begin{align*}
\Delta\left(V\left|\Delta \phi\right|^{p-2} \Delta \phi\right) \geq W \phi^{p-1} \quad \text{and} \quad -\Delta \phi > 0
\end{align*}
a.e. in $\mathbb{R}^N$. Then, there holds
\begin{align}\label{rellich ky}
\int_{\mathbb{R}^N} V\left|\Delta u\right|^p dx \geq \int_{\mathbb{R}^N} W|u|^p dx.
\end{align}
for all real-valued $u\in C_{0}^{\infty}(\mathbb{R}^{N})$. Originally, the result was obtained in a more general setting, namely Baouendi-Grushin (see, Example \ref{ex-grushin}), but for simplicity, we state it in the standard Euclidean one. Later, by a similar technique, Ruzhansky and Sabitbek \cite{ruzhansky2024hardy} extended (\ref{rellich ky}) to all complex-valued $u\in C_{0}^{2}(\Omega)$:
\begin{align}\label{rellich bolys}
\int_{\Omega} V|\Delta| u||^{p} d x \geq \int_{\Omega} W|u|^{p} d x.
\end{align}
Since $|\Delta u|\geq|\Delta |u||$, the inequality (\ref{rellich bolys}) actually gives an improvement to (\ref{rellich ky}).

In this paper, we give a sharp remainder formula for (\ref{rellich ky}) in the setting of general linear second-order differential operators, which include the standard Laplace operator, Baouendi-Grushin, Heisenberg-Greiner and etc.

\section{Preliminaries}

In this section, we aim to provide the necessary notation and facts concerning the general linear second-order differential operator as well as the $C_p$-functional (the key algebraic identity).  

Throughout the paper, we use $\mathcal{I}_{k}$ to denote the identity matrix of order $k$, and $|\cdot|$ to denote the standard Euclidean norm. Let $\sigma:=\left(\sigma_{ij}\right)$, $i=1,\ldots,\ell$, $j=1,\ldots,N$ be an $\ell\times N$ matrix with entries $\sigma_{ij}$ and $\frac{\partial}{\partial x_{j}}\sigma_{ij}$ belonging to $C(\mathbb{R}^N)$. Now define the vector fields $X_{i}$ as
\begin{align*}
X_{i}:=\sum^{N}_{j=1}\sigma_{ij}(x)\frac{\partial}{\partial x_{j}}, \quad i=1,\ldots,\ell.
\end{align*}
and the horizontal gradient $\nabla_{\mathcal{L}}$ as follows
\begin{align*}
\nabla_{\mathcal{L}}:=\sigma\nabla,
\end{align*}
where $\nabla$ is the usual Euclidean gradient in $\mathbb{R}^N$. The assumptions above allow us to naturally define the formal adjoint of $X_{i}$, i.e.
\begin{align*}
X^{*}_{i}:=-\sum_{j=1}^{N}\frac{\partial}{\partial x_{j}}\left(\sigma_{ij}(x) \ \cdot \ \right)
\end{align*}
along with $\nabla^{*}_{\mathcal{L}}:=\left(X^{*}_{1},\ldots,X^{*}_{\ell}\right)$. Now we can define the horizontal divergence 
\begin{align*}
\text{div}_{\mathcal{L}}\left(\cdot\right):=-\nabla^{*}_{\mathcal{L}}\cdot=\text{div}\left(\sigma^{T} \ \cdot \ \right)
\end{align*}
and the linear second-order differential operator
\begin{align*}
\mathcal{L}:=-\nabla^{*}_{\mathcal{L}}\cdot\nabla_{\mathcal{L}}.
\end{align*}
If $1<p<\infty$, we also have the quasilinear second-order degenerate elliptic differential operator defined as
\begin{align*}
\mathcal{L}_{p}u:=-\nabla^{*}_{\mathcal{L}}\cdot\left(|\nabla_{\mathcal{L}}u|^{p-2}\nabla_{\mathcal{L}}u\right).
\end{align*}
The product rule for the horizontal gradient, for two scalar functions $f$ and $g$, is given by
\begin{align*}
\nabla_{\mathcal{L}}(fg)=\nabla_{\mathcal{L}}(f)g+f\nabla_{\mathcal{L}}(g).
\end{align*}
Furthermore, we note the identity below, which will be used in the proof of the $L^{p}$-Rellich identity. 

\begin{prop}\label{prop rellich}
Let $1<p<\infty$ and $\Omega\subseteq \mathbb{R}^N$. Then, for all complex-valued $u$ on $\Omega$ and non-trivial real-valued $\phi\in C^{2}(\Omega)$, we have
\begin{multline}\label{laplace u/phi}
\mathcal{L} \left( \frac{|u|^{p}}{|\phi|^{p-2}\phi} \right) = \frac{1}{|\phi|^{p-2}\phi}\biggl[  \mathcal{L}(|u|^{p})
 - (p-1)\frac{\mathcal{L}\phi}{\phi}\,|u|^{p}
\\ - 2(p-1)\frac{\nabla_{\mathcal{L}}(|u|^{p})\cdot \nabla_{\mathcal{L}}\phi}{\phi}
 + p(p-1)\frac{|\nabla_{\mathcal{L}}\phi|^{2}}{\phi^{2}}\,|u|^{p} \biggr].
\end{multline}   
\end{prop}

In order to complete this part of the section, we examine some examples of second-order differential operators.

\begin{ex}\label{ex-eucl}
If $\sigma=\mathcal{I}_{N}$, then we get the standard gradient $\nabla_{\mathcal{L}}=\nabla$, Laplacian $\mathcal{L}=\Delta$ and the $p$-Laplacian $\mathcal{L}_{p}=\Delta_{p}$.
\end{ex}

\begin{ex}\label{ex-grushin}
Split $\mathbb{R}^N$ in $z=(x,y)\in\mathbb{R}^{m}\times\mathbb{R}^{k}$ with $m,k\geq1$, $m+k=N$ and $\gamma\geq0$. Setting the matrix as
\begin{align*}
\sigma=\left(
\begin{array}{cc}
\mathcal{I}_n & 0 \\
0 & |x|^{\gamma} \mathcal{I}_k
\end{array}
\right),
\end{align*}
we get the sub-elliptic gradient $\nabla_{\mathcal{L}}=\left(\nabla_{x},|x|^{\gamma}\nabla_{y}\right)$, where $\nabla_{x}$ and $\nabla_{y}$ are the gradients in the variables $x$ and $y$ respectively. The linear operator $\mathcal{L}$ is called the Baouendi-Grushin operator $\mathcal{L}=\Delta_{x}+|x|^{2\gamma}\Delta_{y}$. We note that if $\gamma=0$ and $k=0$, then we recover the Euclidean setting in Example \ref{ex-eucl}.
\end{ex}

\begin{ex}
Let $z=(x,y,t)\in\mathbb{R}^{n}\times\mathbb{R}^{n}\times\mathbb{R}$. Then, for $|z|:=|(x,y)|$ and $\gamma\geq 1$, we set
\begin{align*}
\sigma=
\begin{pmatrix}
\mathcal{I}_n & 0 & 2\gamma y |z|^{2\gamma - 2} \\
0 & \mathcal{I}_n & -2\gamma x |z|^{2\gamma - 2}
\end{pmatrix}.
\end{align*}
This gives the vector fields:
\begin{align*}
X_i =\frac{\partial}{\partial x_i} + 2\gamma y_i |z|^{2\gamma - 2} \frac{\partial}{\partial t},
\qquad
Y_i = \frac{\partial}{\partial y_i} - 2\gamma x_i |z|^{2\gamma - 2} \frac{\partial}{\partial t},
\qquad i=1,\ldots,n.
\end{align*}
In the special cases, if $\gamma=1$, then $\mathcal{L}$ coincides with the sub-Laplacian on the Heisenberg group $\mathbb{H}^{n}$. For $\gamma>1$, the second-order differential operator $\mathcal{L}$ becomes the Greiner operator (see, \cite{greiner1979fundamental}). 
\end{ex}

Moving to the next part of this section, we introduce the definition of the $C_p$-functional as an algebraic identity. 

\begin{defn}\label{cp-def}
Let $1<p<\infty$. Then, for $\xi,\eta\in\mathbb{C}^{\ell}$, we define
\begin{align}\label{cp-eq}
C_p(\xi,\eta):=|\xi|^p-|\xi-\eta|^p-p|\xi-\eta|^{p-2}\textnormal{Re}(\xi-\eta)\cdot\overline{\eta}\geq0.
\end{align}
\end{defn}

Here, we note that \cite[Propositions 2.1 and 2.2]{d2025unified} can be described in terms of the $C_p$-functional. Moreover, the identity (\ref{cp-eq}) contains both \cite[Propositions 2.1 and 2.2]{d2025unified} as it already holds true for all complex-valued $\xi,\eta\in \mathbb{C}$ (i.e., $\ell=1$) and general complex vectors $\xi,\eta\in \mathbb{C}^{\ell}$.

\begin{lem}[\text{\cite[Step 3 of Proof of Theorem 1.2]{cazacu2024hardy}}]\label{lem1}
Let $p\geq2$. Then, for $\xi,\eta\in\mathbb{C}^{\ell}$, we have
\begin{align*}
C_{p}(\xi,\eta)\geq c_1(p)|\eta|^{p},
\end{align*}
where 
\begin{align}\label{c1 const}
c_1(p)
= \inf_{(s,t)\in\mathbb{R}^2\setminus\{(0,0)\}}
\frac{\bigl[t^2 + s^2 + 2s + 1\bigr]^{\frac p2} -1-ps}
{\bigl[t^2 + s^2\bigr]^{\frac p2}}\in(0,1].
\end{align}
\end{lem}
\begin{lem}[\text{\cite[Lemma 2.2]{CT24}}]\label{lem2}
Let $1<p<2\leq \ell$. Then, for $\xi,\eta\in \mathbb{C}^{\ell}$, we have
\begin{align*}
C_p(\xi, \eta) \geq c_2(p) \frac{|\eta|^2}{\left( |\xi| + |\xi - \eta| \right)^{2-p}},
\end{align*}
where
\begin{align}\label{c2 const}
c_2(p) := \inf_{s^2 + t^2 > 0} \frac{\left( t^2 + s^2 + 2s + 1 \right)^{\frac{p}{2}} - 1 - ps}{\left( \sqrt{t^2 + s^2 + 2s + 1} + 1 \right)^{p-2} (t^2 + s^2)} \in \left(0,  \frac{p(p-1)}{2^{p-1}} \right].
\end{align}
\end{lem}
\begin{lem}[\text{\cite[Lemma 2.3]{CT24}}]\label{lem3}
Let $1<p<2\leq \ell$. Then, for $\xi,\eta\in \mathbb{C}^{\ell}$, we have
\begin{align*}
C_p(\xi, \eta) \leq c_3(p) \frac{|\eta|^2}{\left( |\xi| + |\xi - \eta| \right)^{2-p}},
\end{align*}
where
\begin{align}\label{c3 const}
c_3(p) := \sup_{s^2 + t^2 > 0} \frac{\left( t^2 + s^2 + 2s + 1 \right)^{\frac{p}{2}} - 1 - ps}{\left( \sqrt{t^2 + s^2 + 2s + 1} + 1 \right)^{p-2} (t^2 + s^2)} \in \left[ \frac{p}{2^{p-1}}, +\infty \right).
\end{align}
\end{lem}

The above lemmata are crucial since they allow us to have the following proposition.

\begin{prop}
$C_p(\xi,\eta)=0$ if and only if $\eta=0$ for $1<p<\infty$ and $\ell \geq 2$.
\end{prop}

\section{Main results}\label{sec: main res}

In this section, using the algebraic identity (\ref{cp-eq}), we show the extended \cite[Theorems 1.1 and 1.2]{d2025unified} to any $1<p<\infty$ and provide a general weighted $L^{p}$-Rellich type inequality. We begin with the extension of \cite[Theorem 1.1]{d2025unified}.

\allowdisplaybreaks

\begin{thm}\label{thm1}
Let $1<p<\infty$, $\Omega \subseteq \mathbb{R}^N$ be an open subset, $Z: \Omega \to \mathbb{R}^{\ell}$, $V(x)\geq 0$ and $W(x)\in L_{loc}^{1}(\Omega)$ such that $V(x)|Z|^{p}\in L_{loc}^{1}(\Omega)$.  Let $\lambda>0$ be a fixed positive constant and assume that there exists a non-trivial real-valued $\phi\in C^{1}(\Omega)$ that solves the following equation in the distributional sense in $\Omega$:
\begin{align}\label{Z eq}
-\textnormal{div}_{\mathcal{L}}\left( V|\nabla_{\mathcal{L}}\phi\cdot Z|^{p-2} \left( \nabla_{\mathcal{L}}\phi \cdot Z \right)Z \right) = 
\lambda W|\phi|^{p-2}\phi.
\end{align}
Then, for all complex-valued $u\in C_{0}^{\infty}(\Omega)$, we have
\begin{align}\label{hardy 1 iden}
\int_{\Omega}V|\nabla_{\mathcal{L}}u\cdot Z|^p dx=\lambda\int_{\Omega}W|u|^{p} dx+\int_{\Omega}C_p\left(V^{\frac{1}{p}}\nabla_{\mathcal{L}}u\cdot Z,V^{\frac{1}{p}}\phi\nabla_{\mathcal{L}}\left(\frac{u}{\phi}\right)\cdot Z\right)dx.
\end{align}
Here, the functions $u=c\phi$ are \say{virtual} extremizers of (\ref{hardy 1 iden}) for all $c\in\mathbb{C}$ and $p\geq 2$.
\end{thm}

\begin{rem}
The extremizers are \say{virtual} in the sense that they make the remainder term vanish, while the integrals appearing in the Hardy identity diverge.
\end{rem}

The same technique allows us to extend the Hardy type inequality in all relevant directions \cite[Theorem 1.2]{d2025unified}.

\begin{thm}\label{thm2}
Let $1<p<\infty$, $\Omega \subseteq \mathbb{R}^N$ be an open subset, $V(x)\geq 0$ and $W(x)$ be two functions in $L_{loc}^{1}(\Omega)$.  Let $\lambda>0$ be a fixed positive constant and assume that there exists a non-trivial real-valued $\phi\in C^{1}(\Omega)$ that solves the following equation in the distributional sense in $\Omega$:
\begin{align*}
-\textnormal{div}_{\mathcal{L}}\left( V|\nabla_{\mathcal{L}}\phi|^{p-2}\nabla_{\mathcal{L}}\phi \right) = \lambda W|\phi|^{p-2}\phi.
\end{align*}
Then, for all complex-valued $u\in C_{0}^{\infty}(\Omega)$, we have
\begin{align}\label{hardy 2 iden}
\int_{\Omega}V|\nabla_{\mathcal{L}}u|^p dx=\lambda\int_{\Omega}W|u|^{p} dx+\int_{\Omega}C_p\left(V^{\frac{1}{p}}\nabla_{\mathcal{L}}u,V^{\frac{1}{p}}\phi\nabla_{\mathcal{L}}\left(\frac{u}{\phi}\right)\right)dx.
\end{align}
Here, the functions $u=c\phi$ are  \say{virtual} extremizers of (\ref{hardy 2 iden}) for all $c\in\mathbb{C}$.
\end{thm}

\begin{rem}
Since $C_{p}(\cdot,\cdot)\geq 0$, the identities (\ref{hardy 1 iden}) and (\ref{hardy 2 iden}) yield the following general weighted $L^{p}$-Hardy type inequalities:
\begin{align*}
\int_{\Omega}V|\nabla_{\mathcal{L}}u\cdot Z|^{p}dx\geq\lambda\int_{\Omega}W|u|^{p}dx.
\end{align*}
and
\begin{align*}
\int_{\Omega}V|\nabla_{\mathcal{L}}u|^{p}dx\geq\lambda\int_{\Omega}W|u|^{p}dx.
\end{align*}
\end{rem}

The next result concerns with the $L^{p}$-Rellich type identity.

\begin{thm}\label{thm rellich}
Let $1<p<\infty$, $\Omega \subseteq \mathbb{R}^N$ be an open subset, $V(x)\geq 0$ and $W$ be two functions in $L^{1}_{loc}(\Omega)$. Let $\lambda$ be a fixed positive constant and assume that there exists a non-trivial real-valued $\phi\in C^{2}(\Omega)$ with $-\frac{\mathcal{L}\phi}{\phi}\geq0$ that solves the following equation in the distributional sense in $\Omega$:
\begin{align*}
\mathcal{L}\left(V|\mathcal{L}\phi|^{p-2}\mathcal{L}\phi\right) = \lambda W|\phi|^{p-2}\phi.
\end{align*}
Then, for all complex-valued $u\in C_{0}^{\infty}(\Omega)$, we have
\begin{align}\label{rellich 1 iden}
\int_{\Omega} V|\mathcal{L} u|^{p} \, dx  \nonumber
&=\lambda\int_{\Omega}^{}W|u|^{p}dx+\int_{\Omega} C_p\left( V^{\frac{1}{p}}\mathcal{L} u,V^{\frac{1}{p}}\mathcal{L} u - V^{\frac{1}{p}}\frac{\mathcal{L} \phi}{\phi}u \right) dx \nonumber
\\&\quad -  p(p-1) \int_{\Omega} V\frac{|\mathcal{L} \phi|^{p-2}\mathcal{L} \phi}{|\phi|^{p-2}\phi}|u|^{p-2}\left| \nabla_{\mathcal{L}}|u|-\frac{|u|}{\phi}\nabla_{\mathcal{L}} \phi \right| ^{2} dx \nonumber 
\\
& \quad - p \int_{\Omega} V\frac{|\mathcal{L} \phi|^{p-2}\mathcal{L} \phi}{|\phi|^{p-2}\phi}|u|^{p-2}\left( |\nabla_{\mathcal{L}} u|^{2}-|\nabla_{\mathcal{L}} |u||^{2} \right) dx.
\end{align}
Here, the functions $u=c\phi$ are \say{virtual} extremizers of (\ref{rellich 1 iden}) for all $c\in \mathbb{C}$.
\end{thm}

\begin{rem}
Likewise, these extremizers are \say{virtual} since they eliminate the remainder terms at the cost of integral divergence.
\end{rem}

\begin{rem}
One can observe that $|\nabla_{\mathcal{L}}u|\geq |\nabla_{\mathcal{L}}|u||$ a.e in $\Omega$. In fact, we can take a non-trivial $u(x) = A(x) + iB(x)$, where \(A=\operatorname{Re} u\) and \(B=\operatorname{Im} u\). Then,
\begin{align*}
|\nabla_{\mathcal{L}}|u||^{2} = \left|\frac{A\nabla_{\mathcal{L}}A+B\nabla_{\mathcal{L}}B}{|u|}\right|^{2}.
\end{align*}
Since $|u|^{2}=A^{2}+B^{2}$, by the Cauchy-Schwarz inequality, we get that
\begin{align*}
|\nabla_{\mathcal{L}}u|\geq|\nabla_{\mathcal{L}}|u||
\end{align*}
a.e. in $\Omega$. We also refer to \cite[Lemma 2.4]{ruzhansky2024hardy} and \cite[Theorem 2.1]{ruzhansky2019weighted2}.
\end{rem}

\begin{rem}
Similarly as for Hardy, since $C_{p}(\cdot,\cdot)\geq0$, $-\frac{\mathcal{L}\phi}{\phi}\geq0$ and $
|\nabla_{\mathcal{L}}u|\geq|\nabla_{\mathcal{L}}|u||$
a.e. in $\Omega$, the identity (\ref{rellich 1 iden}) naturally leads to the general $L^{p}$-Rellich inequality:
\begin{align*}
\int_{\Omega}V|\mathcal{L}u|^{p}dx\geq\lambda\int_{\Omega}W|u|^{p}dx.
\end{align*}
\end{rem}

\section{Applications}

\subsection{Classical implications.} In this section, we present some classical implications of Theorems \ref{thm1}, \ref{thm2} and \ref{thm rellich}. The first of which is the ability to recover sharp remainder formulae of the classical $L^{p}$-Hardy inequality and its improved versions. For simplicity, in (\ref{hardy 1 iden}), let us take
\begin{align*}
\sigma=\mathcal{I}_{N}, \quad V=1, \quad \phi=|x|^{-\frac{N-p}{p}} \text{and} \quad Z=\frac{x}{|x|},
\end{align*}
then, we have the following improved $L^{p}$-Hardy type inequality:

\begin{cor}
Let $1<p<\infty$. Then, for all complex-valued $u\in C_{0}^{\infty}(\mathbb{R}^N \backslash \{0\})$, we have
\begin{multline}\label{improved hardy}
\int_{\mathbb{R}^N}\left|\nabla u \cdot \frac{x}{|x|}\right|^{p}dx = \left(\frac{N-p}{p}\right)^{p}\int_{\mathbb{R}^N}\frac{|u|^{p}}{|x|^{p}}dx \\ + \int_{\mathbb{R}^N}C_{p}\left(\nabla u, |x|^{-\frac{N-p}{p}}\nabla\left(\frac{u}{|x|^{-\frac{N-p}{p}}}\right)\cdot\frac{x}{|x|}\right)dx.
\end{multline}
Moreover, the functions $u=c|x|^{-\frac{N-p}{p}}$ are \say{virtual} extremizers of (\ref{improved hardy}) for all $c\in\mathbb{C}$ and $p\geq2$.
\end{cor}

Different versions of this kind of identity (\ref{improved hardy}) have already appeared in many different works (see, e.g., \cite{lam2021pbessel, kalaman2024cylindrical, yessirkegenov2025refined}). The same choice of functions $V$ and $\phi$, in (\ref{hardy 2 iden}), gives the following $L^{p}$-Hardy identity for complex-valued functions \cite[Lemma 3.3]{cazacu2024hardy}:

\begin{cor}
Let $1<p<\infty$. Then, for all complex-valued $u\in C_{0}^{\infty}(\mathbb{R}^N\backslash \{0\})$, we have
\begin{align}\label{hardy class}
\int_{\mathbb{R}^N}|\nabla u|^{p}dx=\left(\frac{N-p}{p}\right)^{p}\int_{\mathbb{R}^N}\frac{|u|^{p}}{|x|^{p}}dx + \int_{\mathbb{R}^N}C_{p}\left(\nabla u, |x|^{-\frac{N-p}{p}}\nabla\left(\frac{u}{|x|^{-\frac{N-p}{p}}}\right)\right)dx.
\end{align}
Moreover, the functions $u=c|x|^{-\frac{N-p}{p}}$ are \say{virtual} extremizers of (\ref{hardy class}) for all $c\in \mathbb{C}$.
\end{cor}

We also obtain some new identities for the $L^{p}$-Rellich inequality. In particular, by taking
\begin{align*}
\sigma=\mathcal{I}_{N}, \quad  V=1 \quad \text{and} \quad \phi=|x|^{-\alpha}, \quad \text{where} \quad \alpha=\frac{N-2p}{p}, 
\end{align*}
in (\ref{rellich 1 iden}), we obtain the following sharp remainder term of the classical $L^{p}$-Rellich inequality:

\begin{cor}\label{cor rellich classic}
Let $1<p\leq\frac{N}{2}$. Then, for all complex-valued $u\in C_{0}^{\infty}(\mathbb{R}^N\backslash\{0\})$, we have
\begin{multline}\label{rellich identity}
\int_{\mathbb{R}^N}|\Delta u|^{p}dx=C^{p}_{N,p}\int_{\mathbb{R}^N}\frac{|u|^{p}}{|x|^{2p}}dx+\int_{\mathbb{R}^N}C_{p}\left(\Delta u, \Delta u + \frac{C_{N,p}}{|x|^{2}}\right)dx
\\ + p(p-1)C^{p-1}_{N,p}\int_{\mathbb{R}^N}\frac{|u|^{p-2}}{|x|^{2(p-1)}}\left|\nabla |u|+\frac{N-2p}{p}|u|\frac{x}{|x|^{2}}\right|^{2}dx 
\\ + pC^{p-1}_{N,p}\int_{\mathbb{R}^N}\frac{|u|^{p-2}}{|x|^{2(p-1)}}\left(|\nabla u|^{2}-|\nabla |u||^{2}\right)dx,
\end{multline}
where $C_{N,p}=\frac{N(p-1)(N-2p)}{p^{2}}$. Moreover, the functions $u=c|x|^{-\frac{N-2p}{p}}$ are \say{virtual} extremizers of (\ref{rellich identity}) for all $c\in\mathbb{C}$.
\end{cor}

\begin{rem}
After dropping the remainder terms, we obtain the $L^{p}$-Rellich inequality with a sharp constant:
\begin{align*}
\int_{\mathbb{R}^N}|\Delta u|^{p}dx\geq \left(\frac{N(p-1)(N-2p)}{p^{2}}\right)^{p}\int_{\mathbb{R}^N}\frac{|u|^{p}}{|x|^{2p}}dx.
\end{align*}
\end{rem}

The case $p=2$ of (\ref{rellich identity}) also seems to be new.

\begin{cor}
Let $N\geq 4$. Then, for all complex-valued $u\in C_{0}^{\infty}(\mathbb{R}^{N}\backslash\{0\})$, we have
\begin{multline}\label{rellich l2}
\int_{\mathbb{R}^N}|\Delta u|^{2}dx=C^{2}_{N}\int_{\mathbb{R}^N}\frac{|u|^{2}}{|x|^{4}}dx+\int_{\mathbb{R}^N}\left|\Delta u+\frac{C_{N}}{|x|^{2}}\right|^{2}dx
\\+2C_{N}\int_{\mathbb{R}^N}\frac{1}{|x|^{2}}\left|\nabla|u|+\frac{N-4}{4}|u|\frac{x}{|x|^{2}}\right|^{2}dx
\\+2C_{N}\int_{\mathbb{R}^N}\frac{|\nabla u|^{2}-|\nabla |u||^{2}}{|x|^{2}}dx,
\end{multline}
where $C_{N}=\frac{N(N-4)}{4}$. Moreover, the functions $u=c|x|^{-\frac{N-4}{2}}$ are \say{virtual} extremizers of (\ref{rellich l2}) for all $c\in\mathbb{C}$.
\end{cor}

\begin{rem}
In the same spirit, we obtain the classical $L^{2}$-Rellich inequality as a consequence of (\ref{rellich l2}):
\begin{align*}
\int_{\mathbb{R}^N}|\Delta u|^{2}dx\geq\left(\frac{N(N-4)}{4}\right)^{2}\int_{\mathbb{R}^N}\frac{|u|^{2}}{|x|^{4}}dx.
\end{align*}
\end{rem}

\begin{rem}
The above results are almost identical in structure to the general case for the $\mathcal{L}$ operator. However, we state the Euclidean ones for simplicity and to emphasize the novelty even for the standard Laplacian $\Delta$ in the case of the Rellich inequality.
\end{rem}

\subsection{Extension of Cossetti and D’Arca results and more.} As \cite[Applications I-III]{d2025unified} are derived from \cite[Theorems 1.1 and 1.2]{d2025unified}, the validity of those results for all $1<p<\infty$ directly implies that each of the subsequent applications also carries over to the entire range of $p\in(1,\infty)$. 

Let $d: \Omega \to \mathbb{R}$ be a non-negative and continuous function such that $\nabla_{\mathcal{L}}d \neq 0$ a.e. in $\Omega$ with the following properties:

\begin{enumerate}
    \item There exists $x_{0} \in \Omega$ such that $d \in C^{\infty}(\Omega \setminus \{x_{0}\})$.
    
    \item $r := \inf_{x \in \Omega} d(x)$ and $R := \sup_{x \in \Omega} d(x)$, where $R$ can take the value $+\infty$.
    
    \item $V(d) \geq 0$ and $W(d)$ are radial weights defined on the interval $I = (r, R)$.
\end{enumerate}

Additionally, we assume \cite[H.0-H.3]{d2025unified}. Under these assumptions, we have the following Corollary as a direct implication of Theorem \ref{thm1} (see, \cite[Section 3]{d2025unified} for more details).

\begin{cor}\label{cor 1}
Let $1<p<\infty$ and $\phi(d)$ be a solution to (\ref{Z eq}). Then, for all complex-valued $u\in C_{0}^{\infty}(\Omega)$, we have
\begin{multline}\label{cor 1 eq}
\int_{\Omega} V(d)\left|\nabla_{\mathcal{L}} u \cdot \frac{\nabla_{\mathcal{L}} d}{\left|\nabla_{\mathcal{L}} d\right|}\right|^{p} d x = \lambda \int_{\Omega} W(d)|u|^{p}\left|\nabla_{\mathcal{L}} d\right|^{p} d x \\+ \int_{\Omega}C_{p}\left(V^{\frac{1}{p}}\nabla_{\mathcal{L}}u\cdot\frac{\nabla_{\mathcal{L}} d}{\left|\nabla_{\mathcal{L}} d\right|},V^{\frac{1}{p}}\phi(d)\nabla_{\mathcal{L}}\left(\frac{u}{\phi(d)}\right)\cdot\frac{\nabla_{\mathcal{L}} d}{\left|\nabla_{\mathcal{L}} d\right|}\right)dx.
\end{multline}
Moreover, the functions $u=c\phi(d)$ are \say{virtual} extremizers of (\ref{cor 1 eq}) for all $c\in\mathbb{C}$ and $p\geq 2$.
\end{cor}

\begin{rem}
After dropping the remainder term and restricting $p\geq2$, in (\ref{cor 1 eq}), we obtain \cite[Theorem 3.1]{d2025unified}:
\begin{align}\label{d'arca 1st cor}
\int_{\Omega} V(d)\left|\nabla_{\mathcal{L}} u \cdot \frac{\nabla_{\mathcal{L}} d}{\left|\nabla_{\mathcal{L}} d\right|}\right|^{p} d x \geq \lambda \int_{\Omega} W(d)|u|^{p}\left|\nabla_{\mathcal{L}} d\right|^{p} d x.
\end{align}
As we said initially, the authors, in \cite[Section 3]{d2025unified}, then apply (\ref{d'arca 1st cor}) to derive both old and new inequalities \cite[Section 3]{d2025unified}. The identity (\ref{cor 1 eq}) provides an explicit sharp remainder term and extends all of the subsequent results in \cite[Section 3]{d2025unified} to the entire range of $1<p<\infty$. In fact, the same applies for \cite[Section 4-5]{d2025unified}. 
\end{rem}

There is also one more interesting application of Theorem \ref{thm2}, which is not present in \cite{d2025unified}. To be more precise, let us consider an open bounded set $\Omega$ of $\mathbb{R}^N$ and the following eigenvalue problem:
\begin{align}\label{p-laplacian quas}
\begin{cases} 
-\mathcal{L}_{p} u = \lambda|u|^{p-2} u & \text{in } \Omega, \\ 
u = 0 & \text{on } \partial \Omega,
\end{cases}
\end{align}
where any non-trivial weak solution to (\ref{p-laplacian quas}) is called an eigenfunction of the degenerate Dirichlet $p$-Laplacian $\mathcal{L}_{p}u=-\nabla^{*}_{\mathcal{L}}\cdot\left(|\nabla_{\mathcal{L}}u|^{p-2}\nabla_{\mathcal{L}}u\right)$ with a corresponding eigenvalue $\lambda\in \mathbb{R}$. We can also define the first eigenvalue as
\begin{align}\label{eigenvalue first quas}
\lambda_{1}(p, \Omega, \mathcal{L}) = \inf \left\{ \frac{\int_{\Omega} |\nabla_{\mathcal{L}} u|^{p} dx}{\int_{\Omega} |u|^{p} dx} : u \in W_{0}^{1,p}(\Omega) \setminus \{0\} \right\}.
\end{align}
and the associated eigenfunction we denote as $u_{1}$. Thus, having this in mind, choosing
\begin{align*}
V=W=1, \quad \lambda=\lambda_{1}(p, \Omega, \mathcal{L}) \quad \text{and} \quad \phi=u_{1}
\end{align*}
in Theorem \ref{thm2}, we get the following general identity:

\begin{cor}
Let $1<p<\infty$ and $\Omega$ be an open bounded subset of $\mathbb{R}^N$. Then, for all complex-valued $u\in C_{0}^{\infty}(\Omega)$, we have
\begin{align}\label{quas poincare}
\int_{\Omega}|\nabla_{\mathcal{L}}u|^{p}dx = \lambda_{1}(p,\Omega,\mathcal{L})\int_{\Omega}|u|^{p}+\int_{\Omega}C_{p}\left(\nabla_{\mathcal{L}}u,u_{1}\nabla_{\mathcal{L}}\left(\frac{u}{u_{1}}\right)\right)dx.
\end{align}
Moreover, the functions $u=cu_{1}$ are \say{virtual} extremizers of (\ref{quas poincare}) for all $c\in\mathbb{C}$.
\end{cor}

\begin{rem}
In the special case, identity (\ref{quas poincare}) recovers known $L^{p}$-Poincar\'e identities such as those in \cite{ozawa2020sharp, suragan2023sharp, apseit2025sharp} for the standard Euclidean and general Baouendi-Grushin cases.
\end{rem}

\section{Proofs of Theorems \ref{thm1} and \ref{thm2}}

This section is devoted to the proofs of Theorems \ref{thm1} and \ref{thm2}.

\begin{proof}[\textbf{Proof of Theorem \ref{thm1}}]
We know that
\begin{align*}
C_{p}(\xi,\eta)=|\xi|^{p}-|\xi-\eta|^{p}-p|\xi-\eta|^{p-2}\text{Re}\left( \xi-\eta \right) \cdot \overline{\eta}.
\end{align*}
for $1<p<\infty$ and $\xi,\eta \in \mathbb{C}^{n}$. Let $\xi:=f$ and $\eta:=f-g$. Then
\begin{align}\label{Cp in terms of f and g}
C_{p}(f,f-g)=|f|^{p}+(p-1)|g|^{p}-p\text{Re}|g|^{p-2}g\overline{f}. 
\end{align}
Now take $\Omega \subseteq \mathbb{R}^{N}$, complex-valued $u \in C_{0}^{\infty}(\Omega)$, $V(x)\geq 0$, $W(x)\in L_{loc}^{1}(\Omega)$, $Z:\Omega \to \mathbb{R}^{\ell}$ and a non-trivial real-valued $\phi \in C^{1}(\Omega)$. Define
\begin{align}\label{def of f and g}
f:=V^{\frac{1}{p}}\nabla_{\mathcal{L}}u\cdot Z \in L^{p}\left(\Omega\right) \quad \text{and} \quad g:=\frac{\nabla_{\mathcal{L}}\phi \cdot Z}{\phi}uV^{\frac{1}{p}}\in L^{p}(\Omega).
\end{align}
By direct computation, we get
\begin{align}
p \text{Re}\int_{\Omega}|g|^{p-2}g\overline{f}dx&= p\text{Re}\int_{\Omega}\frac{V|\nabla_{\mathcal{L}}\phi \cdot Z|^{p-2}\nabla_{\mathcal{L}}\phi \cdot Z}{|\phi|^{p-2}\phi}|u|^{p-2}u\overline{\nabla_{\mathcal{L}}u}\cdot Z dx \nonumber
\\&=\int_{\Omega}^{}\nabla_{\mathcal{L}}|u|^{p}\cdot\frac{V|\nabla_{\mathcal{L}}\phi \cdot Z|^{p-2}\nabla_{\mathcal{L}}\phi \cdot Z}{|\phi|^{p-2}\phi}Z dx \nonumber
\\&=\int_{\Omega}^{}\nabla_{\mathcal{L}}\left( \frac{|u|^{p}}{|\phi|^{p-2}\phi} \right)\cdot \left( V|\nabla_{\mathcal{L}}\phi \cdot Z|^{p-2}\nabla_{\mathcal{L}}\phi \cdot Z \right)Z dx \nonumber
\\& \quad + (p-1)\int_{\Omega}^{}|u|^{p}\frac{|\nabla_{\mathcal{L}}\phi \cdot Z|^{p}}{|\phi|^{p}}V dx. \label{ggf}
\end{align}
Let us assume that there exists a non-trivial real-valued $\phi \in C^{1}(\Omega)$ that solves the following equation in the distributional sense in $\Omega$:
\begin{align*}
-\text{div}_{\mathcal{L}}\left( V|\nabla_{\mathcal{L}}\phi\cdot Z| \left( \nabla_{\mathcal{L}}\phi \cdot Z \right)Z \right) = 
\lambda W|\phi|^{p-2}\phi,
\end{align*}
which is equivalent to
\begin{align}\label{weak sense}
\int_{\Omega}^{}V|\nabla_{\mathcal{L}}\phi \cdot Z|^{p-2}\left( \nabla_{\mathcal{L}}\phi \cdot Z \right) \left( \nabla_{\mathcal{L}}\varphi \cdot Z \right) dx = \lambda\int_{\Omega}^{}W|\phi|^{p-2}\phi \ \varphi dx 
\end{align}
for all $\varphi \in C^{\infty}_{c}(\Omega)$. Substituting $\varphi=\frac{|u|^{p}}{|\phi|^{p-2}\phi}$ to (\ref{weak sense}), we have that
\begin{align}
\int_{\Omega}^{}V|\nabla_{\mathcal{L}}\phi \cdot Z|^{p-2}\left( \nabla_{\mathcal{L}}\phi \cdot Z \right) \left( \nabla_{\mathcal{L}}\left( \frac{|u|^{p}}{|\phi|^{p-2}\phi} \right)\cdot Z  \right) dx =\lambda\int_{\Omega}^{}W|u|^{p} dx. \label{using dist}
\end{align}
Using (\ref{using dist}) in (\ref{ggf}), we get
\begin{align}
p \text{Re}\int_{\Omega}|g|^{p-2}g\overline{f}dx=\lambda\int_{\Omega}^{}W|u|^{p} dx+(p-1)\int_{\Omega}^{}|u|^{p}\frac{|\nabla_{\mathcal{L}}\phi \cdot Z|^{p}}{|\phi|^{p}}V dx. \label{last step thm}
\end{align}
Now let us integrate (\ref{Cp in terms of f and g}) over $\Omega$:
\begin{align}
\int_{\Omega}C_{p}(f,f-g) dx=\int_{\Omega}|f|^{p} dx + (p-1)\int_{\Omega}|g|^{p} dx - p\text{Re}\int_{\Omega}|g|^{p-2}g\overline{f} dx. \label{last last step}
\end{align}
Putting (\ref{def of f and g}) and (\ref{last step thm}) to (\ref{last last step}), we obtain
\begin{multline*}
\int_{\Omega}C_p\left(V^{\frac{1}{p}}\nabla_{\mathcal{L}}u\cdot Z,V^{\frac{1}{p}}\nabla_{\mathcal{L}}u\cdot Z-\frac{\nabla_{\mathcal{L}}\phi \cdot Z}{\phi}uV^{\frac{1}{p}}\right)dx=\int_{\Omega}V|\nabla_{\mathcal{L}}u\cdot Z|^p dx \\+ (p-1)\int_{\Omega}|u|^p\frac{|\nabla_{\mathcal{L}}\phi\cdot Z|^p}{|\phi|^p}V dx  - \lambda\int_{\Omega}W|u|^{p} dx - (p-1)\int_{\Omega}|u|^{p}\frac{|\nabla_{\mathcal{L}}\phi \cdot Z|^p}{|\phi|^p}V dx
\\ = \int_{\Omega}V|\nabla_{\mathcal{L}}u\cdot Z|^p dx - \lambda\int_{\Omega}W|u|^{p} dx.
\end{multline*}
Dropping the remainder term gives us
\begin{align}\label{ineq}
\int_{\Omega}V|\nabla_{\mathcal{L}}u\cdot Z|^p dx \geq \lambda\int_{\Omega}W|u|^{p} dx.
\end{align}
Finally, since
\begin{align*}
C_{p}(\xi,\eta)=0 \iff \eta=f-g=0
\end{align*}
and
\begin{align*}
f-g=V^{\frac{1}{p}}\nabla_{\mathcal{L}}u\cdot Z-\frac{\nabla_{\mathcal{L}}\phi \cdot Z}{\phi}uV^{\frac{1}{p}}=V^{\frac{1}{p}}\phi\nabla_{\mathcal{L}}\left(\frac{u}{\phi}\right)\cdot Z=0,
\end{align*}
we have that the remainder term vanishes (or that the inequality (\ref{ineq}) achieves equality) if and only if $u=c\phi$ for any $c\in\mathbb{C}$.
\end{proof}

\begin{proof}[\textbf{Proof of Theorem \ref{thm2}}]
Let $u$ and $V,\phi$ be as in the assumptions of Theorem \ref{thm2}. Taking 
\begin{align}\label{new fg}
f:=V^{\frac{1}{p}}\nabla_{\mathcal{L}}u \in L^{p}\left(\Omega\right) \quad \text{and} \quad g:=\frac{\nabla_{\mathcal{L}}\phi }{\phi}uV^{\frac{1}{p}}\in L^{p}(\Omega)
\end{align}
and proceeding in the same way as in the proof of Theorem \ref{thm1}, we get
\begin{align*}
\int_{\Omega} C_{p}\!\left(
V^{\frac{1}{p}}\nabla_{\mathcal{L}}u,
V^{\frac{1}{p}}\nabla_{\mathcal{L}}u
-\frac{\nabla_{\mathcal{L}}\phi}{\phi}u V^{\frac{1}{p}}
\right)dx
&=  \int_{\Omega}V|\nabla_{\mathcal{L}}u|^{p}dx 
- \lambda\int_{\Omega}W|u|^{p}dx.
\end{align*}
Since $C_{p}\ge 0$, we get the inequality
\begin{align}\label{ineq thm2}
\int_{\Omega}V|\nabla_{\mathcal{L}}u|^{p}dx \ge \lambda\int_{\Omega}W|u|^{p}dx.
\end{align}
Finally, equality in \eqref{ineq thm2} holds iff
\[
V^{\frac{1}{p}}\nabla_{\mathcal{L}}u
-\frac{\nabla_{\mathcal{L}}\phi}{\phi}u V^{\frac{1}{p}} = 0.
\]
That is,
\[
V^{\frac{1}{p}}\phi\,\nabla_{\mathcal{L}}\!\left(\frac{u}{\phi}\right) = 0,
\]
which gives $u=c\phi$ for some constant $c\in\mathbb{C}$.
\end{proof}

\section{Proofs of Proposition \ref{prop rellich}, Theorem \ref{thm rellich} and Corollary \ref{cor rellich classic}}

First, let us provide the proof of Proposition \ref{prop rellich}. Then, we will proceed with Theorem \ref{thm rellich}.

\begin{proof}[\textbf{Proof of Proposition \ref{prop rellich}}]
We are given that $1<p<\infty$, $\Omega \subseteq \mathbb{R}^N$. Now let us take a complex-valued $u$ on $\Omega$ and a non-trivial real-valued $\phi\in C^{2}(\Omega)$. Then, by direct calculations, we get
\begin{align*}
\mathcal{L} \left( \frac{|u|^{p}}{|\phi|^{p-2}\phi} \right) = \mathcal{L}\left(v\Phi\right),
\end{align*}
where
\begin{align*}
v = |u|^{p} \quad \text{and} \quad \Phi = \frac{1}{|\phi|^{p-2}\phi}.
\end{align*}
Then, by the product rule, we get
\begin{align*}
\mathcal{L}(v\Phi) = \mathcal{L}(v)\Phi+v\mathcal{L}(\Phi)+2\nabla_{\mathcal{L}}v\cdot\nabla_{\mathcal{L}}\Phi.
\end{align*}
Let us calculate each part term by term. We start with $\mathcal{L}(v)\Phi$:
\begin{align*}
\mathcal{L}(v)\Phi = \mathcal{L}(|u|^{p})\Phi.
\end{align*}
To evaluate the rest, we note that
\begin{align*}
|\phi|^{p}\Phi = \phi.
\end{align*}
Therefore, we have
\begin{align*}
\mathcal{L}\left(|\phi|^{p}\Phi\right)&=\mathcal{L}(|\phi|^{p})\Phi+|\phi|^{p}\mathcal{L}(\Phi)+2\nabla_{\mathcal{L}}|\phi|^{p}\cdot\nabla_{\mathcal{L}}\Phi
\\&=\left[p(p-1)|\phi|^{p-2}|\nabla_{\mathcal{L}} \phi|^{2}+p|\phi|^{p-2}\phi\mathcal{L}\phi\right]\Phi+|\phi|^{p}\mathcal{L}(\Phi)+2\nabla_{\mathcal{L}}|\phi|^{p}\cdot\nabla_{\mathcal{L}}\Phi
\\&=p(p-1)\frac{|\nabla_{\mathcal{L}}\phi|^2}{\phi}+p\mathcal{L}\phi+|\phi|^{p}\mathcal{L}(\Phi)+2\nabla_{\mathcal{L}}|\phi|^{p}\cdot\nabla_{\mathcal{L}}\Phi
\\&=p(p-1)\frac{|\nabla_{\mathcal{L}}\phi|^2}{\phi}+p\mathcal{L}\phi+|\phi|^{p}\mathcal{L}(\Phi)
\\& \quad + 2\left(p|\phi|^{p-2}\phi\nabla_{\mathcal{L}}\phi\right)\cdot\left((1-p)|\phi|^{-p}\nabla_{\mathcal{L}}\phi\right)
\\&=-p(p-1)\frac{|\nabla_{\mathcal{L}}\phi|^{2}}{\phi}+p\mathcal{L}\phi+|\phi|^{p}\mathcal{L}\Phi=\mathcal{L}\phi.
\end{align*}
Since $\phi\neq 0$, this gives that
\begin{align*}
\mathcal{L}\Phi = \Phi\left[-(p-1)\frac{\mathcal{L}\phi}{\phi}+(p-1)\frac{|\nabla_{\mathcal{L}}\phi|^{2}}{\phi^{2}}\right]
\end{align*}
Combining all pieces together, we get
\begin{align*}
\mathcal{L}\left(v\Phi\right) &= \mathcal{L}(|u|^p)\Phi + |u|^p\Phi \left[ -(p-1)\frac{\mathcal{L}\phi}{\phi} + p(p-1)\frac{|\nabla_{\mathcal{L}}\phi|^2}{\phi^2} \right] 
\\& \quad - 2(p-1)\Phi \frac{\nabla_{\mathcal{L}}(|u|^p) \cdot \nabla_{\mathcal{L}}\phi}{\phi},
\end{align*}
which after factoring $\Phi$, gives us (\ref{laplace u/phi}).
\end{proof}

\begin{proof}[\textbf{Proof of Theorem \ref{thm rellich}}]
Let us recall the identity (\ref{laplace u/phi}):
\begin{multline*}
\mathcal{L} \left( \frac{|u|^{p}}{|\phi|^{p-2}\phi} \right) = \frac{1}{|\phi|^{p-2}\phi}\biggl[  \mathcal{L}(|u|^{p})
 - (p-1)\frac{\mathcal{L}\phi}{\phi}\,|u|^{p}
\\ - 2(p-1)\frac{\nabla_{\mathcal{L}}(|u|^{p})\cdot \nabla_{\mathcal{L}}\phi}{\phi}
 + p(p-1)\frac{|\nabla_{\mathcal{L}}\phi|^{2}}{\phi^{2}}\,|u|^{p} \biggr] .
\end{multline*}
Multiplying by $|\mathcal{L} \phi|^{p-2}\mathcal{L} \phi$, gives us
\begin{align*}
&\mathcal{L}\!\left(\frac{|u|^{p}}{|\phi|^{p-2}\phi}\right)\, |\mathcal{L}\phi|^{p-2}\mathcal{L}\phi \nonumber
\\& \quad = 
\frac{|\mathcal{L}\phi|^{p-2}\mathcal{L}\phi}{|\phi|^{p-2}\phi}\,\mathcal{L}(|u|^{p})
- (p-1)\frac{|\mathcal{L}\phi|^{p-2}\mathcal{L}\phi}{|\phi|^{p-2}\phi}\,
\frac{\mathcal{L} \phi}{\phi}\,|u|^{p} \nonumber
\\& \quad \quad - 2(p-1)\frac{|\mathcal{L}\phi|^{p-2}\mathcal{L}\phi}{|\phi|^{p-2}\phi}\,
\frac{\nabla_{\mathcal{L}}(|u|^{p})\cdot \nabla_{\mathcal{L}}\phi}{\phi}
+ p(p-1)\frac{|\mathcal{L}\phi|^{p-2}\mathcal{L}\phi}{|\phi|^{p-2}\phi}\,
\frac{|\nabla_{\mathcal{L}}\phi|^{2}}{\phi^{2}}\,|u|^{p}. 
\end{align*}
Thus, we have
\begin{align}\label{multiply}
&\frac{|\mathcal{L}\phi|^{p-2}\mathcal{L}\phi}{|\phi|^{p-2}\phi}\,\mathcal{L}(|u|^{p}) \nonumber
\\& \quad = \mathcal{L}\!\left(\frac{|u|^{p}}{|\phi|^{p-2}\phi}\right)\, |\mathcal{L}\phi|^{p-2}\mathcal{L}\phi+(p-1)\frac{|\mathcal{L}\phi|^{p-2}\mathcal{L}\phi}{|\phi|^{p-2}\phi}\,
\frac{\mathcal{L} \phi}{\phi}\,|u|^{p} \nonumber
\\& \quad \quad + 2(p-1)\frac{|\mathcal{L}\phi|^{p-2}\mathcal{L}\phi}{|\phi|^{p-2}\phi}\,
\frac{\nabla_{\mathcal{L}}(|u|^{p})\cdot \nabla_{\mathcal{L}}\phi}{\phi} - p(p-1)\frac{|\mathcal{L}\phi|^{p-2}\mathcal{L}\phi}{|\phi|^{p-2}\phi}\,
\frac{|\nabla_{\mathcal{L}}\phi|^{2}}{\phi^{2}}\,|u|^{p}.
\end{align}
On the other hand, similarly as in the proofs of Theorem \ref{thm1} and \ref{thm2}, by direct calculation, we get 
\begin{align}\label{cp rellich}
&C_p\left(V^{\frac{1}{p}} \mathcal{L} u,V^{\frac{1}{p}}\mathcal{L} u - V^{\frac{1}{p}}\frac{\mathcal{L} \phi}{\phi}u \right) \nonumber
\\& \quad=V|\mathcal{L} u|^{p}
+ (p-1)\,V\frac{|\mathcal{L}\phi|^{p}}{|\phi|^{p}}\,|u|^{p} - p\operatorname{Re}\!\left(
 V\frac{|\mathcal{L}\phi|^{p-2}\mathcal{L}\phi}{|\phi|^{p-2}\phi}\,
 |u|^{p-2}u\,\overline{\mathcal{L} u}
  \right).
\end{align}
Using the identity
\begin{align*}
-p|u|^{p-2} \operatorname{Re}(u \overline{\mathcal{L} u}) = - \mathcal{L} (|u|^p) + p|u|^{p-2} |\nabla_{\mathcal{L}} u|^2 + p(p-2)|u|^{p-4} |\operatorname{Re}(u \overline{\nabla_{\mathcal{L}} u})|^2 
\end{align*}
in (\ref{cp rellich}), we get
\begin{align}\label{before sub cp}
&C_p\left(V^{\frac{1}{p}} \mathcal{L} u,V^{\frac{1}{p}}\mathcal{L} u - V^{\frac{1}{p}}\frac{\mathcal{L} \phi}{\phi}u \right) \nonumber
\\& \quad=V|\mathcal{L} u|^{p}
+ (p-1)\,V\frac{|\mathcal{L}\phi|^{p}}{|\phi|^{p}}\,|u|^{p}
-V\frac{|\mathcal{L} \phi|^{p-2}\mathcal{L} \phi}{|\phi|^{p-2}\phi}\mathcal{L}(|u|^{p}) \nonumber
\\& \quad \quad+pV\frac{|\mathcal{L} \phi|^{p-2}\mathcal{L} \phi}{|\phi|^{p-2}\phi}|u|^{p-2}|\nabla_{\mathcal{L}} u|^{2}
+p(p-2)V\frac{|\mathcal{L} \phi|^{p-2}\mathcal{L} \phi}{|\phi|^{p-2}\phi}|u|^{p-4}\left| \text{Re}\left( u\overline{\nabla_{\mathcal{L}} u} \right) \right|^{2} . 
\end{align}
Putting (\ref{multiply}) in (\ref{before sub cp}), we obtain
\begin{align*}
&C_p\left( V^{\frac{1}{p}}\mathcal{L} u,V^{\frac{1}{p}}\mathcal{L} u - V^{\frac{1}{p}}\frac{\mathcal{L} \phi}{\phi}u \right) \\& \quad=V|\mathcal{L} u|^{p}
+ (p-1)\,V\frac{|\mathcal{L}\phi|^{p}}{|\phi|^{p}}\,|u|^{p}
 - V\mathcal{L} \left( \frac{|u|^{p}}{|\phi|^{p-2}\phi} \right)|\mathcal{L} \phi|^{p-2}\mathcal{L} \phi 
 \\& \quad\quad-(p-1)V\frac{|\mathcal{L}\phi|^{p-2}\mathcal{L}\phi}{|\phi|^{p-2}\phi}\frac{\mathcal{L} \phi}{\phi}|u|^{p}
-2(p-1)V\frac{|\mathcal{L}\phi|^{p-2}\mathcal{L}\phi}{|\phi|^{p-2}\phi}\frac{\nabla_{\mathcal{L}} \left( |u|^{p} \right) \cdot \nabla_{\mathcal{L}} \phi}{\phi}
\\& \quad\quad+p(p-1)V\frac{|\mathcal{L}\phi|^{p-2}\mathcal{L}\phi}{|\phi|^{p-2}\phi}\frac{|\nabla_{\mathcal{L}} \phi|^{2}}{\phi^{2}}|u|^{p}
 +pV\frac{|\mathcal{L} \phi|^{p-2}\mathcal{L} \phi}{|\phi|^{p-2}\phi}|u|^{p-2}|\nabla_{\mathcal{L}} u|^{2}
\\& \quad\quad+p(p-2)V\frac{|\mathcal{L} \phi|^{p-2}\mathcal{L} \phi}{|\phi|^{p-2}\phi}|u|^{p-4}\left| \text{Re}\left( u\overline{\nabla_{\mathcal{L}} u} \right) \right|^{2} .
\end{align*}
Simplifying:
\begin{align*}
&C_p\left( V^{\frac{1}{p}}\mathcal{L} u,V^{\frac{1}{p}}\mathcal{L} u - V^{\frac{1}{p}}\frac{\mathcal{L} \phi}{\phi}u \right) \\&\quad=V|\mathcal{L} u|^{p}
- V\mathcal{L} \left( \frac{|u|^{p}}{|\phi|^{p-2}\phi} \right)|\mathcal{L} \phi|^{p-2}\mathcal{L} \phi -2(p-1)V\frac{|\mathcal{L}\phi|^{p-2}\mathcal{L}\phi}{|\phi|^{p-2}\phi}\frac{\nabla_{\mathcal{L}} \left( |u|^{p} \right) \cdot \nabla_{\mathcal{L}} \phi}{\phi}
\\& \quad\quad +p(p-1)V\frac{|\mathcal{L}\phi|^{p-2}\mathcal{L}\phi}{|\phi|^{p-2}\phi}\frac{|\nabla_{\mathcal{L}} \phi|^{2}}{\phi^{2}}|u|^{p}
+pV\frac{|\mathcal{L} \phi|^{p-2}\mathcal{L} \phi}{|\phi|^{p-2}\phi}|u|^{p-2}|\nabla_{\mathcal{L}} u|^{2}
\\& \quad \quad +p(p-2)V\frac{|\mathcal{L} \phi|^{p-2}\mathcal{L} \phi}{|\phi|^{p-2}\phi}|u|^{p-4}\left| \text{Re}\left( u\overline{\nabla_{\mathcal{L}} u} \right) \right|^{2} .
\end{align*}
Let $G:=V\frac{|\mathcal{L} \phi|^{p-2}\mathcal{L} \phi}{|\phi|^{p-2}\phi}$. Since
\begin{align*}
\text{Re}\left( u\overline{\nabla_{\mathcal{L}} u} \right)=|u|\nabla_{\mathcal{L}} (|u|) \quad \text{and} \quad \nabla_{\mathcal{L}}|u|^{p}=p|u|^{p-2}\text{Re}\left( u\overline{\nabla_{\mathcal{L}} u} \right) =p|u|^{p-2}|u|\nabla_{\mathcal{L}} (|u|),
\end{align*}
we get
\begin{align*}
&C_p\left( V^{\frac{1}{p}}\mathcal{L} u,V^{\frac{1}{p}}\mathcal{L} u - V^{\frac{1}{p}}\frac{\mathcal{L} \phi}{\phi}u \right) \\&\quad= V|\mathcal{L} u|^{p}
- V\mathcal{L} \left( \frac{|u|^{p}}{|\phi|^{p-2}\phi} \right)|\mathcal{L} \phi|^{p-2}\mathcal{L} \phi  + p(p-1)G\frac{|\nabla_{\mathcal{L}} \phi|^{2}}{\phi^{2}}|u|^{p}
\\& \quad\quad-2p(p-1)G|u|^{p-2}\left( \nabla_{\mathcal{L}} |u|\cdot \frac{|u|}{\phi}\nabla_{\mathcal{L}} \phi \right) 
+pG|u|^{p-2} |\nabla_{\mathcal{L}} u|^{2}
\\&\quad\quad+p(p-2)G|u|^{p-2}\left| \nabla_{\mathcal{L}}|u| \right|^{2}
\\&\quad=V|\mathcal{L} u|^{p}
- V\mathcal{L} \left( \frac{|u|^{p}}{|\phi|^{p-2}\phi} \right)|\mathcal{L} \phi|^{p-2}\mathcal{L} \phi + p(p-1)G\frac{|\nabla_{\mathcal{L}} \phi|^{2}}{\phi^{2}}|u|^{p}
\\& \quad\quad -2p(p-1)G|u|^{p-2}\left( \nabla_{\mathcal{L}} |u|\cdot \frac{|u|}{\phi}\nabla_{\mathcal{L}} \phi \right) 
+ pG|u|^{p-2}\left| \nabla_{\mathcal{L}}|u| \right| ^{2}
\\&\quad\quad+pG|u|^{p-2}\left( |\nabla_{\mathcal{L}} u|^{2}-|\nabla_{\mathcal{L}} |u||^{2} \right)
+ p(p-2)G|u|^{p-2}|\nabla_{\mathcal{L}} |u||^{2}
\\&\quad = V|\mathcal{L} u|^{p}
- V\mathcal{L} \left( \frac{|u|^{p}}{|\phi|^{p-2}\phi} \right)|\mathcal{L} \phi|^{p-2}\mathcal{L} \phi + p(p-1)G\frac{|\nabla_{\mathcal{L}} \phi|^{2}}{\phi^{2}}|u|^{p}
\\& \quad\quad -2p(p-1)G|u|^{p-2}\left( \nabla_{\mathcal{L}} |u|\cdot \frac{|u|}{\phi}\nabla_{\mathcal{L}} \phi \right)
+p(p-1)G|u|^{p-2}|\nabla_{\mathcal{L}} |u||^{2}
\\& \quad\quad+pG|u|^{p-2}\left( |\nabla_{\mathcal{L}} u|^{2}-|\nabla_{\mathcal{L}} |u||^{2} \right)
\\&\quad=V|\mathcal{L} u|^{p}
- V\mathcal{L} \left( \frac{|u|^{p}}{|\phi|^{p-2}\phi} \right)|\mathcal{L} \phi|^{p-2}\mathcal{L} \phi
\\& \quad\quad +p(p-1)G|u|^{p-2}\left[ |\nabla_{\mathcal{L}} |u||^{2}-2\nabla_{\mathcal{L}}|u|\cdot \frac{|u|}{\phi}\nabla_{\mathcal{L}} \phi + \frac{|u|^{2}}{\phi^{2}}|\nabla_{\mathcal{L}} \phi|^{2} \right] \\&\quad\quad+pG|u|^{p-2}\left( |\nabla_{\mathcal{L}} u|^{2}-|\nabla_{\mathcal{L}} |u||^{2} \right)
\\&\quad=V|\mathcal{L} u|^{p}
- V\mathcal{L} \left( \frac{|u|^{p}}{|\phi|^{p-2}\phi} \right)|\mathcal{L} \phi|^{p-2}\mathcal{L} \phi+p(p-1)G|u|^{p-2}\left| \nabla_{\mathcal{L}}|u|-\frac{|u|}{\phi}\nabla_{\mathcal{L}} \phi \right| ^{2}
\\& \quad\quad  +pG|u|^{p-2}\left( |\nabla_{\mathcal{L}} u|^{2}-|\nabla_{\mathcal{L}} |u||^{2} \right).
\end{align*}
Making the change of variables $G=V\frac{|\mathcal{L} \phi|^{p-2}\mathcal{L} \phi}{|\phi|^{p-2}\phi}$, we have
\begin{align*}
&C_p\left( V^{\frac{1}{p}}\mathcal{L} u,V^{\frac{1}{p}}\mathcal{L} u - V^{\frac{1}{p}}\frac{\mathcal{L} \phi}{\phi}u \right)\\&\quad=V|\mathcal{L} u|^{p}
- V\mathcal{L} \left( \frac{|u|^{p}}{|\phi|^{p-2}\phi} \right)|\mathcal{L} \phi|^{p-2}\mathcal{L} \phi
\\&\quad\quad+p(p-1)V\frac{|\mathcal{L} \phi|^{p-2}\mathcal{L} \phi}{|\phi|^{p-2}\phi}|u|^{p-2}\left| \nabla_{\mathcal{L}}|u|-\frac{|u|}{\phi}\nabla_{\mathcal{L}} \phi \right| ^{2}
\\&\quad\quad+pV\frac{|\mathcal{L} \phi|^{p-2}\mathcal{L} \phi}{|\phi|^{p-2}\phi}|u|^{p-2}\left( |\nabla_{\mathcal{L}} u|^{2}-|\nabla_{\mathcal{L}} |u||^{2} \right).
\end{align*}
Integrating over $\Omega$, we obtain
\begin{align*}
&\int_{\Omega} C_p\left( V^{\frac{1}{p}}\mathcal{L} u,V^{\frac{1}{p}}\mathcal{L} u - V^{\frac{1}{p}}\frac{\mathcal{L} \phi}{\phi}u \right) dx \\&\quad= \int_{\Omega}V|\mathcal{L} u|^{p} \, dx - \int_{\Omega} V\mathcal{L} \left( \frac{|u|^{p}}{|\phi|^{p-2}\phi} \right)|\mathcal{L} \phi|^{p-2}\mathcal{L} \phi dx \\
& \quad\quad + p(p-1) \int_{\Omega} V\frac{|\mathcal{L} \phi|^{p-2}\mathcal{L} \phi}{|\phi|^{p-2}\phi}|u|^{p-2}\left| \nabla_{\mathcal{L}}|u|-\frac{|u|}{\phi}\nabla_{\mathcal{L}} \phi \right| ^{2}dx \\
& \quad\quad + p \int_{\Omega} V\frac{|\mathcal{L} \phi|^{p-2}\mathcal{L} \phi}{|\phi|^{p-2}\phi}|u|^{p-2}\left( |\nabla_{\mathcal{L}} u|^{2}-|\nabla_{\mathcal{L}} |u||^{2} \right) dx.
\end{align*}
Applying integration by parts and the fact that there exists a non-trivial real-valued $\phi \in C^{2}\left(\Omega\right)$ that solves the following equation in the distributional sense in $\Omega$:
\begin{align*}
\mathcal{L} \left( V|\mathcal{L} \phi|^{p-2}\mathcal{L} \phi \right)=\lambda W|\phi|^{p-2}\phi, 
\end{align*}
we obtain
\begin{align*}
&\int_{\Omega} C_p\left( V^{\frac{1}{p}}\mathcal{L} u,V^{\frac{1}{p}}\mathcal{L} u - V^{\frac{1}{p}}\frac{\mathcal{L} \phi}{\phi}u \right)dx \\&\quad= \int_{\Omega} V|\mathcal{L} u|^{p} dx -\lambda\int_{\Omega}^{}W|u|^{p}dx \\&\quad\quad+ p(p-1) \int_{\Omega} V\frac{|\mathcal{L} \phi|^{p-2}\mathcal{L} \phi}{|\phi|^{p-2}\phi}|u|^{p-2}\left| \nabla_{\mathcal{L}}|u|-\frac{|u|}{\phi}\nabla_{\mathcal{L}} \phi \right| ^{2}dx \\
& \quad\quad + p \int_{\Omega} V\frac{|\mathcal{L} \phi|^{p-2}\mathcal{L} \phi}{|\phi|^{p-2}\phi}|u|^{p-2}\left( |\nabla_{\mathcal{L}} u|^{2}-|\nabla_{\mathcal{L}} |u||^{2} \right)dx.
\end{align*}
Rewriting
\begin{align*}
&\int_{\Omega} V|\mathcal{L} u|^{p} dx \\&\quad=\lambda\int_{\Omega}^{}W|u|^{p}dx+\int_{\Omega} C_p\left( V^{\frac{1}{p}}\mathcal{L} u,V^{\frac{1}{p}}\mathcal{L} u - V^{\frac{1}{p}}\frac{\mathcal{L} \phi}{\phi}u \right) dx
\\&\quad\quad -  p(p-1) \int_{\Omega} V\frac{|\mathcal{L} \phi|^{p-2}\mathcal{L} \phi}{|\phi|^{p-2}\phi}|u|^{p-2}\left| \nabla_{\mathcal{L}}|u|-\frac{|u|}{\phi}\nabla_{\mathcal{L}} \phi \right| ^{2} dx \\
& \quad\quad - p \int_{\Omega} V\frac{|\mathcal{L} \phi|^{p-2}\mathcal{L} \phi}{|\phi|^{p-2}\phi}|u|^{p-2}\left( |\nabla_{\mathcal{L}} u|^{2}-|\nabla_{\mathcal{L}} |u||^{2} \right) dx.
\end{align*}
Finally, we set $-\frac{\mathcal{L} \phi}{\phi}>0$ and use the fact that $|\nabla_{\mathcal{L}} u|\geq|\nabla_{\mathcal{L}}|u||$ (see, \cite[Lemma 2.4]{ruzhansky2024hardy} and \cite[Theorem 2.1]{ruzhansky2019weighted2}) to complete the proof.
\end{proof}

\begin{proof}[\textbf{Proof of Corollary \ref{cor rellich classic}}]
We have that $\sigma = \mathcal{I}_{k}$, $V=1$ and $\phi=|x|^{-\alpha}$ with $\alpha=\frac{N-2p}{p}$. By direct calculations, we have
\begin{align*}
\phi'(r) = -\alpha r^{-\alpha-1},
\qquad
\phi''(r) = \alpha(\alpha+1) r^{-\alpha-2}.
\end{align*}
Hence,
\begin{align*}
\Delta \phi
= \phi'' + \frac{n-1}{r}\phi'
= \alpha(\alpha+1) r^{-\alpha-2}
+ (n-1)(-\alpha r^{-\alpha-2})
= \alpha(\alpha-n+2) r^{-\alpha-2}.
\end{align*}
Let us write this as follows:
\begin{align*}
\Delta \phi = c_1 r^{-\alpha-2},
\qquad
c_1 := \alpha(\alpha-n+2).
\end{align*}
Simplifying, we get
\begin{align*}
c_{1} = -A := \frac{n(p-1)(n-2p)}{p^2} > 0,
\qquad (n>2p).
\end{align*}
Since $\Delta \phi = -A r^{-\alpha-2}$ and $A>0$,
\begin{align*}
|\Delta \phi|^{p-2}\Delta \phi
= (A r^{-\alpha-2})^{p-2}(-A r^{-\alpha-2})
= -A^{p-1} r^{-(\alpha+2)(p-1)}.
\end{align*}
Setting
\begin{align*}
g(r) := |\Delta \phi|^{p-2}\Delta \phi
= K r^{-\beta},
\qquad
K := -A^{p-1},
\qquad
\beta := (\alpha+2)(p-1)=\frac{n(p-1)}{p}.
\end{align*}
The calculations for $\Delta g$ are the same as for $\Delta \phi$:
\begin{align*}
\Delta g = \beta(\beta-n+2) K r^{-\beta-2},
\end{align*}
which means that
\begin{align*}
\Delta\bigl(|\Delta\phi|^{p-2}\Delta\phi\bigr)
= \beta(\beta-n+2) K r^{-\beta-2}.
\end{align*}
Computing the constants, we obtain
\begin{align*}
\beta(\beta-n+2)K
= (-A)(-A^{p-1})
= A^p
\end{align*}
and thus
\begin{align*}
\Delta\bigl(|\Delta\phi|^{p-2}\Delta\phi\bigr)
= A^p r^{-\beta-2}.
\end{align*}
Since
\begin{align*}
-\beta-2 = -2p-\alpha(p-1)
\end{align*}
and
\begin{align*}
|\phi|^{p-2}\phi = r^{-\alpha(p-1)},
\end{align*}
we finally derive that
\begin{align*}
\Delta\bigl(|\Delta\phi|^{p-2}\Delta\phi\bigr)
= A^{p}|x|^{-2p} |\phi|^{p-2}\phi.
\end{align*}
Therefore,
\begin{align*}
\lambda W = A^{p}|x|^{-2p}.
\end{align*}
Other calculations follow directly:
\begin{align*}
\frac{\nabla \phi}{\phi}=-\alpha\frac{x}{|x|^{2}}, \quad \frac{\Delta \phi}{\phi}=-A\frac{1}{|x|^{2}}, \quad \frac{|\Delta \phi|^{p-2}\Delta\phi}{|\phi|^{p-2}\phi}=-A^{p-1}\frac{1}{|x|^{2(p-1)}}.
\end{align*}
Combining everything together, we get (\ref{rellich identity}).
\end{proof}

\bibliographystyle{alpha}
\bibliography{citation}

\end{document}